\documentclass[a4paper,oneside,reqno,english]{amsart}
\usepackage[utf8]{inputenc}
\usepackage[T1]{fontenc}

\DeclareFontFamily{OMX}{mlmex}{}
\DeclareFontShape{OMX}{mlmex}{m}{n}{%
   <->mlmex10%
   }{}%
\usepackage{mlmodern}

\usepackage[hscale=0.666, vscale=0.7]{geometry}
\usepackage{babel}
\usepackage[numbers,sort&compress]{natbib}

\usepackage{hyperref}
\hypersetup{%
  colorlinks=true,%
  citecolor=[RGB]{120,29,126},%
  pdfauthor={Jean-François Burnol},%
  pdfsubject={Zeta series with missing digits},%
  pdfstartview=FitH,%
  pdfpagemode=UseNone,%
}

\usepackage{centeredline}
\usepackage{booktabs}
\usepackage{amsmath,amsthm,amssymb}
\usepackage{mathtools}
\DeclarePairedDelimiterX\Iffint[2]{\lbrack\!\lbrack}{\rbrack\!\rbrack}{#1...#2}
\DeclarePairedDelimiterX\Ioo[2]{\lparen}{\rparen}{#1,#2}
\DeclarePairedDelimiterX\Iof[2]{\lparen}{\rbrack}{#1,#2}
\DeclarePairedDelimiterX\Ifo[2]{\lbrack}{\rparen}{#1,#2}
\DeclarePairedDelimiterX\Iff[2]{\lbrack}{\rbrack}{#1,#2}
\newcommand\sD{\mathsf{D}}
\newcommand\NN{\mathbb{N}}
\newcommand\ZZ{\mathbb{Z}}
\newcommand\RR{\mathbb{R}}
\newcommand\CC{\mathbb{C}}
\newcommand\cW{\mathcal{W}}

\newcommand\cB{\mathcal{B}}
\newcommand\e{\mathsf{e}}
\newcommand\dmu{\mathrm{d}\mu}
\newcommand\dt{\mathrm{d}t}
\newcommand\du{\mathrm{d}u}
\newcommand\dx{\mathrm{d}x}

\allowdisplaybreaks

\newcommand\arxivurl[1]{\href{https://arxiv.org/abs/#1}{\textsf{arXiv:#1}}}

\newtheorem{theorem}{Theorem}[section]

\newtheorem{lemma}[theorem]{Lemma}
\newtheorem{proposition}[theorem]{Proposition}
\theoremstyle{definition}
\newtheorem{definition}[theorem]{Definition}
\newtheorem{remark}[theorem]{Remark}

\title[Series for $\zeta(s)$]{%
  Some series representing the zeta function for \lowercase{\boldmath$\Re(s)>1$}}

\author[J.-F. Burnol]{Jean-François Burnol}

\address{Université de Lille\\
  Faculté des Sciences et technologies\\
  Département de mathématiques\\
  Cité Scientifique\\
  F-59655 Villeneuve d'Ascq cedex\\
  France
}
\email{jean-francois.burnol@univ-lille.fr}


\date{July 20, 2026. This preprint updates the abstract, references, and the
  numbering of propositions to match the version to appear in
  \href{https://www.aimsciences.org/FCNT}{Frontiers in Combinatorics and
    Number Theory}. Despite the impression one gathers from the title and
  abstract, the bulk of the paper is devoted to the general zeta series with
  missing digits. The specialization to the
  Riemann zeta function is done only in the introductory section and in the
  last one. See \cite{burnolzeta_umosc} for the continued study of the
  case of the Riemann zeta function.}

\begin{document}
\begin{abstract}
  We represent the Riemann zeta function in the half-plane $\Re s >1$ via
  series whose terms admit geometrically decreasing bounds. This is based upon
  simple combinatorics of certain measures, and their associated moments, which
  are defined on words from an alphabet of digits in a given radix $b$.
\end{abstract}

\maketitle

\vspace*{-2\baselineskip}
{\footnotesize
\section*{Numbering of propositions}

\centeredline{\begin{tabular}{ll|ll|ll}
\toprule
  \emph{here and FCNT}&\emph{earlier versions}&
  \emph{here and FCNT}&\emph{earlier versions}&
  \emph{here and FCNT}&\emph{earlier versions}\\
\midrule
  Definition 2.1&   Definition 1&     Lemma 2.8 &   Lemma 2 &              Defintion 3.1 &   Defintion 4 \\
  Definition 2.2&   Definition 2&     Proposition 2.9 &   Proposition 4 &  Proposition 3.2 &   Proposition 6 \\
  Lemma 2.3&   Lemma 1&               Lemma 2.10 &   Lemma 3 &             Proposition 3.3 &   Proposition 7 \\
  Definition 2.4&   Definition 3&     Proposition 2.11 &   Proposition 5&  Theorem 4.1 &   Theorem 2 \\
  Proposition 2.5&   Proposition 1&   Theorem 2.12 &   Theorem 1&          Remark 4.2 &   Remark 3  \\
  Proposition 2.6&   Proposition 2&   Remark 2.13 &   Remark 1\\
  Proposition 2.7&   Proposition 3&   Remark 2.14 &   Remark 2\\
\bottomrule
\end{tabular}}
\bigskip
}

\section{Main result}

In this section, we specialize to $b=2$ the Theorem \ref{thm:conclusion} which
concludes this paper and applies to all integers $b>1$.

Let $s\in\CC$, $\Re s>1$. Define coefficients $u_m^*(s)$ via a linear recurrence:
\begin{align}
\notag
  u_0^*(s) &= \frac{2^s}{2^s-2},
\\
\label{eq:rec}
  m\geq1&\implies u_m^*(s) = \frac1{2^{s+m}-2}
          \sum_{j=1}^{m} \frac{(s+m)\dots(s+m-j+1)}{j!}u_{m-j}^*(s).
\end{align}
Let $\ell$ be an integer at least $2$.  Then
  \begin{equation}\label{eq:1}
    \begin{split}
      \zeta(s) = \sum_{0<n<2^{\ell-1}} \frac1{n^s} &+ \frac{2^{s}}{2^{s}
        -2}\sum_{2^{\ell-1}\leq n<2^{\ell}} \frac1{n^s}
      \\
      &+ \sum_{m=1}^\infty (-1)^m \frac {s\,u_m^*(s)}{s+m}
      (\sum_{2^{\ell-1}\leq n<2 ^{\ell}} \frac1{n^{s+m}}).
    \end{split}
  \end{equation}
  The series with alternating signs at the end of \eqref{eq:1} is absolutely
  convergent.  For $s$ real, its partial sums provide lower and upper bounds
  depending only on parity of number of terms.
  
Explicitly, with $\ell=2$:
\begin{equation*}
      \zeta(s) = 1 
      +  \frac{2^{s}}{2^{s} -2}\bigl(\frac1{2^s} + \frac1{3^s}\bigr)\\
      + \sum_{m=1}^\infty (-1)^m \frac {s\,u_m^*(s)}{s+m}\bigl(\frac1{2^{s+m}} + \frac1{3^{s+m}}\bigr).
\end{equation*}
With $\ell=3$, and \emph{the same} $u_m^*(s)$:
\begin{equation*}
  \begin{split}
    \zeta(s) = 1 &+ 2^{-s} + 3^{-s}
    +  \frac{2^{s}}{2^{s} -2}\bigl(4^{-s}+5^{-s}+6^{-s}+7^{-s}\bigr)\\
    &+ \sum_{m=1}^\infty (-1)^m \frac {s\,u_m^*(s)}{s+m}
       \bigl(4^{-s-m}+5^{-s-m}+6^{-s-m}+7^{-s-m}\bigr).
  \end{split}
\end{equation*}
And similarly for $\ell=4$, $5$, \dots.  Previous experience from
\cite{burnolkempner,burnolirwin} suggests that the formulas are well suited to
obtain hundreds of digits, at least for real $s$.  We
tested briefly using Python double float type, and numerical results matched
(up to the unavoidable rounding errors) with those provided from using the
\texttt{mpmath.zeta()} function at a higher precision.

Due to the linear cost intrinsic to the recurrence \eqref{eq:rec}
(and the fact that this recurrence impedes parallelization), the formula
\eqref{eq:1} will probably not be of much interest for obtaining tens of
thousands of digits or more.  Perhaps in relation to this, these series may
well prove to be new, at least we found no mention of anything similar in
references such as \cite{borwbradcran2000}.

A priori bounds allow to estimate how many terms will be enough for a given
target precision:
\begin{align*}
s=\sigma>1:  \frac{1}{2^\sigma-2} &<  u_m^*(\sigma) \leq \frac{2^\sigma}{2^\sigma-2}\;,\\
s=\sigma+it:
  |u_m^*(s)|&\leq \frac{\Gamma(\sigma +1)}{|\Gamma(s+1)|} \frac{2^\sigma}{|2^s-2|}\;.
 \end{align*}
 This proves that the series in \eqref{eq:1} converges geometrically (with a
 bonus $O(m^{-1})$ factor) with ratio $1/2^{\ell-1}$, also for a complex $s$,
 $\Re s=\sigma>1$.  The exponential increase of $1/|\Gamma(s)|$ on vertical
 lines suggests that large imaginary parts are a challenge to numerical
 implementation of the series, and indeed (see the discussion after Theorem
 \ref{thm:conclusion}), up to a factor of $2^\sigma$ the above upper bound
 gives at least for some $s$ with large imaginary parts a good idea of the
 order of magnitude of $|u_m^*(s)|$.

Defining $u_m(s)$ such that
\begin{equation*}
  u_m^*(s) = \frac{(s+1)_m}{m!} u_m(s) = \frac{(s+1)\dots(s+m)}{m!} u_m(s),
\end{equation*}
it turns out that $|2^s - 2||u_m(s)| \leq (2^\sigma-2)u_m(\sigma)$, $\Re s =
\sigma>1$, from which the above upper bound is derived, and that for $m\geq1$:
\begin{equation}\label{eq:ber2}
  u_m(s) =  \frac1{m+1}\frac{2^{s}}{2^s -2} - \frac{2^{s+1}}{2(2^{s+1}-2)} +
  \sum_{1\leq k \leq \lfloor \frac m2 \rfloor}
   \frac{m!}{(m-2k+1)!}\frac{B_{2k}}{(2k)!}\frac{2^{s+2k}}{2^{s+2k}-2}\;.
\end{equation}
(So, $u_m(s)$ is $2\pi i/\log(2)$ periodic). But, this formula is much less
appropriate for numerical computations than \eqref{eq:rec}: for example with
$s=3$ and $m=30$ some individual summands prove to be larger by many orders of
magnitude than $u_m(s)$, and computations using the \textsf{IEEE-754} ``double
float'' type showed spectacular loss of precision.

The above expression gives the meromorphic continuation of the $u_m(s)$
coefficients to the whole complex plane.  In particular, the linear recurrence
\eqref{eq:rec}, if started with the initial value $1$ at $m=0$, makes sense
for $\Re s >0$. Especially with $s=1$, the solution to this recurrence is the
constant sequence $(1)_{m\geq0}$, so it suggests:
\begin{equation}
  \lim_{s\to1}\frac{2^s -2}{2^s}\zeta(s) =
  \sum_{m=0}^\infty (-1)^m\sum_{2^{\ell-1}\leq n<2^\ell} \frac{1}{(m+1)n^m}\;.
\end{equation}
On the right-hand side, we recognize $\sum_{2^{\ell-1}\leq n<2^\ell}
\log\frac{n+1}n = \log 2$. Thus, the identity obtained at $s=1$ is indeed
correct... The analytic continuation of \eqref{eq:1} to the complex plane
is the object of \cite{burnolzeta_umosc}, which uses the present work
as starting point.

The paper is organized as follows.  In the first section, we extend our
earlier framework \cite{burnolkempner}, which involves certain measures and
moments tailored to compute harmonic series with restrictions on the allowed
digits.  For a given radix $b$ and choice of admissible digits, we develop the
apparatus to compute the associated Dirichlet series not only at $s=1$ but
everywhere in the half-plane of convergence (such ``restricted'' Dirichlet
series have been considered earlier
\cite{kohlspil2009,nathanson2021jnt,alloshalstip2025}).

In the second section we study the exponential moment generating function, and
establish some of its analytic properties.  These analytic aspects are related
to an old theme going back all the way to Hardy
\cite{hardy1907,mendsebb1999,keatread2000,
  balamendsebb2005,zhang2011,zhang2023}.

In the final section, we remove all restrictions on the radix-$b$ digits, so
that we are computing the Riemann zeta function $\zeta(s)$, $\Re s>1$.  From
the point of view of obtaining ``explicit'' formulas, the removal of all
restrictions brings a substantial theoretical simplification and allows to
express the series coefficients in terms of Bernoulli numbers as is shown in
equation \eqref{eq:ber2} for $b=2$.  But, as pointed out already, these
``explicit'' formulas are mostly of theoretical interest, the recurrence
\eqref{eq:rec} being numerically much more efficient than equation
\eqref{eq:ber2}, despite having roughly twice as many terms.

\section{Measures on words, and moments}

Let $b>1$ an integer. Let $\cW$ be the space of words (also called finite
strings) on the alphabet $\sD=\Iffint{0}{b-1}$ of the $b$-ary digits.  The
empty word is noted $\epsilon$. The length of a word $w$ is noted $\ell(w)$.
The power notation $w^j$ means concatenation ($w^2=ww$, etc\dots). And
$w^0=\epsilon$.
\begin{definition}\label{def:1}
  For $w=d_1\dots d_l\in \cW$ we say that $w$ ``represents'' the integer
  $n(w)= d_1b^{l-1} + \dots + d_l \in \NN$.  In particular $n(\epsilon)=0$.
  We define $x(w)= n(w)/b^{\ell(w)}$.  In other terms, $x(w)$ is obtained in
  $b$-radix representation from positioning the word $w$ immediately to the
  right of the radix separator. It belongs to $\Ifo01$.

  For $n$ a non-negative integer, $l=\ell(n)$ is the length of the
  minimal-length representation of $n$ by a word, i.e.\@ it is the smallest
  non-negative integer such that $n<b^l$.

  Let $A$ be some subset of $\sD$.  Its elements will be called the
  ``admissible'' digits, and words using only them are said to be
  ``admissible''.  The admissible integers are those whose minimal
  $b$-representation is an admissible word.  So $0$ is always an admissible
  integer, but is not necessarily an admissible digit.
\end{definition}
Throughout this section it is assumed that the set $A$ of admissible digits is
neither empty nor reduced to the singleton $\{0\}$.  So the set $x(\cW)$
contains positive elements.  Contrarily to \cite{burnolkempner} it is now
allowed for $A$ to be the full alphabet of digits.  We let $N$ be the
cardinality of $A$ and $N_1$ the one of $A\setminus\{0\}$. So $0< N_1\leq
N\leq b$.

\begin{definition}
  Let $s\in\RR$.  We define a positive measure $\mu_{A,s}$ on $\cW$.
  \begin{itemize}
  \item If the word $w$ is admissible,  $\mu_{A,s}(\{w\})= b^{-s\ell(w)}$.
  \item Else, $\mu_{A,s}(\{w\})= 0$.
  \end{itemize}
  For $s$ complex non real, such that the positive measure $\mu_{A,\Re s}$ is
  finite (see next), we use the same formulas. So $\mu_{A,s}$ is a complex
  measure and its measure of  variations is $\mu_{A,\Re s}$.
\end{definition}
Observe that $\mu_{A,s}(\{\epsilon\}) = 1$, whenever $\mu_{A,s}$ is defined,
and that $\mu_{A,s+2\pi i/\log(b)} = \mu_{A,s}$ if $\mu_{A,s}$ is a complex measure.

We compute the total mass for real $s$:
\begin{equation*}
  \mu_{A,s}(\cW) = \sum_{l=0}^\infty b^{-sl}N^l = 
  \begin{cases}
  \frac{b^s}{b^s -N} &(s>\log_b(N)),\\
  +\infty&(s\leq \log_b(N)).
  \end{cases}
\end{equation*}
Thus, $\mu_{A,s}$ is well-defined as a complex measure if and only if
$\mu_{A,\Re s}$ is finite if and only if $\Re s > \log_b(N)$.  Let us record
this formula for the total ``mass'' (but this terminology is a bit misleading
if $s$ is non-real) for $s$ in the open half-plane of convergence:
\begin{equation}\label{eq:u0}
  \mu_{A,s}(\cW) = \frac{b^s}{b^s -N}.
\end{equation}
Except for the sole case with $A=\sD$, one has $N<b$, $\log_b(N)<1$, so then $s=1$
is in the open half-plane of those complex numbers to which we have associated
a complex measure on the space of words $\cW$.

We record the following observation (which is not
used elsewhere) for real $s$:
\begin{lemma}
  For any non-negative bounded function $\phi$ on $\cW$, such that there is at
  least one admissible non-empty word $w$ with $\phi(w)>0$, the function
  $g(\sigma)= \int_{\cW} \phi(w)d\mu_{A,\sigma}(w)$ is positive, strictly
  decreasing, and log-convex on $\Ioo{\log_b(N)}{\infty}$.
\end{lemma}
\begin{proof}
  The quantity $g(\sigma)$ is, up to the additive constant $\phi(\epsilon)$, a
  positive combination of strictly decreasing log-convex functions
  $\sigma\mapsto b^{-l\sigma}$, for some indices $l\geq1$ at least.
\end{proof}

\begin{definition}\label{def:um}
  The ``moments''  $u_{A,s}(m)$ of the complex measure $\mu_{A,s}$ are:
  \begin{equation*}
    u_{A,s}(m) = \int_{\cW} x(w)^m \dmu_{A,s}(w).
  \end{equation*}
\end{definition}
Note that the moments are invariant under $s\mapsto s + 2\pi i/\log(b)$.
Observe further, using that $A$ is neither empty nor reduced to $\{0\}$, the strict
inequality:
\begin{equation*}
  \Im s\notin \frac{2\pi\ZZ}{\log b}\implies |u_{A,s}(m)|< u_{A,\Re s}(m).
\end{equation*}
Indeed, if $a\neq0$ is any admissible digit it contributes in particular
$a^mb^{-m-s}+(ba+a)^mb^{-2m-2s}$ and the modulus of this can be equal to the
sum of the moduli of the two summands only if their ratio is a positive real
number, i.e.\@ if and only if $b^s>0$.  We obtain in Proposition
\ref{prop:bound}, stated later, a (still relatively trivial) quantitative
improvement on the above ``trivial'' estimate.

As we have assumed that $A$ is
neither empty nor reduced to the sole digit $0$, the measure is not supported
only on those $w$'s such that $x(w)=0$.  Consequently, if $s$ is real (and
$>\log_b N$), the sequence of moments is strictly decreasing.  And for every
complex $s$ with $\Re(s)>\log_b N$, the moments converge to zero as
$m\to\infty$, from dominated convergence (or elementary arguments).

\begin{proposition}\label{prop:recu}
  The moments of $\mu_{A,s}$ obey the following recurrence relation:
\begin{equation}\label{eq:recu}
  m\geq1\implies
  (b^{m+s} - N)u_{A,s}(m) = \sum_{j=1}^{m} \binom{m}{j}(\sum_{a\in A}a^j) u_{A,s}(m-j).
\end{equation}
\end{proposition}
\begin{proof}
  We decompose non-empty words as a one-digit prefix $a$ and a suffix of
  length one less.  The ``weight'' of $aw$ is $b^{-s}$ times the weight of
  $w$.  In the next equations $\delta_0$ is the Dirac delta function.  Let
  $m\geq0$:
\begin{align*}
  u_{A,s}(m) &= \delta_0(m) + \sum_{a\in A}\int_{\cW} x(aw)^m b^{-s}\dmu_{A,s}(w),\\
            &= \delta_0(m) + \sum_{a\in A}\int_{\cW} b^{-m}\bigl(a + x(w)\bigr)^m b^{-s}\dmu_{A,s}(w),\\
  u_{A,s}(m) &= \delta_0(m) + b^{-m-s}\sum_{j=0}^m \binom{m}{j}(\sum_{a\in A}a^j) u_{A,s}(m-j).
\end{align*}
For $m=0$ we recover $u_{A,s}(0) = \frac{b^s}{b^s-N}$.  For $m>0$, we obtain
\eqref{eq:recu}.
\end{proof}
Let us record here:
\begin{align}\label{eq:u1}
  u_{A,s}(1) &= \frac{b^s\sum_{a\in A} a}{(b^s -N)(b^{s+1}-N)}\;,\\
\label{eq:u2}
  u_{A,s}(2) &= \frac{b^s\left(2(\sum_{a\in A} a)^2 + 
                (b^{s+1}-N)\sum_{a\in A} a^2\right)}{(b^s -N)(b^{s+1}-N)(b^{s+2}-N)}\;.
\end{align}

The recurrence formula allows a quantitative improvement upon the
``trivial'' estimate $|u_{A,s}(m)|\leq u_{A,\sigma}(m)$, $\sigma=\Re s$.
\begin{proposition}\label{prop:bound}
  Let $s=\sigma+it$, $\sigma>\log_b N$. There holds:
  \begin{equation*}
    |u_{A,s}(m)|\leq \frac{b^\sigma - N}{|b^s - N|} u_{A,\sigma}(m).
  \end{equation*}
\end{proposition}
\begin{proof}
  Define $c_m(s) = \frac{u_{A,s}(m)}{u_{A,s}(0)}$.  The sequence
  $(c_m(s))_{m\geq0}$ obeys the same recurrence \eqref{eq:recu} as $(u_m(s))$
  save for the initial term, which is now $1$.  Observe that $\bigl|b^{m+s} - N\bigr|\geq
  b^{m+\sigma}-N$.  As the binomial coefficients and the power sums which
  appear in the recurrence \eqref{eq:recu} are all non-negative, we conclude
  by induction that $|c_m(s)|\leq c_m(\sigma)$ for every $m\geq0$.
\end{proof}

Consider now the following ``restricted'' zeta value, where the $'$ sign indicates
that only those integers are kept, which are $A$-admissible:
\begin{equation}\label{eq:K}
  K_{b,A,s} = \sideset{}{'}\sum_{n>0} \frac1{n^s}\;.
\end{equation}
We have incorporated $b$ into the subscript, it should arguably have been used
earlier also with the measure and the moments.  We choose a \emph{level}
$\ell\geq1$, and consider the contributions from those admissible integers
having at least $\ell$ digits. We gather them according to the
\emph{admissible} word $w$ of their first $\ell$ digits, so that the minimal
representation of $n$ is $wz$ with some word $z$, possibly empty.
\begin{align}
\notag  n &= b^{|z|}\bigl(b^{|w|}x(w)+ x(z)\bigr),\\
\label{eq:startswithw}
\sideset{}{'}\sum_{n\text{ starts with }w} \frac1{n^s}
&=
  \int_{z\in\cW} \bigl(n(w)+ x(z)\bigr)^{-s}\dmu_{A,s}(z).
\end{align}
Caution that here we do not allow $w=\epsilon$, else, we would need to
restrict the integral to those $z$'s not starting with a $0$ (i.e.\@ such that
$x(z)\geq b^{-1}$), which would then indeed give a formula for the full
$K_{b,A,s}$, but we do not use it here.  In the process leading to the formula
\eqref{eq:startswithw} with $s\in \RR$ arbitrary, everything is non-negative
and the re-writing of the Dirichlet series as an integral against a positive
measure $\mu_{A,s}$ is valid for all $s\in\RR$.  As any function $x\mapsto
(a+x)^{-s}$ on $\Ifo01$ with $a>0$ has a positive minimum and is bounded, its
integral against a positive measure is finite if and only if that measure is
finite.  So the restricted harmonic series \eqref{eq:K} converges if and only
if the push-forward of $\mu_{A,s}$ under the map $x:\cW\to\Ifo01$ is a finite
measure, i.e.\@ if and only if $\mu_{A,s}(\cW)<\infty$, i.e.\@ if and only if
$s>\log_b N$.  See \cite{kohlspil2009,nathanson2021jnt,alloshalstip2025} for
prior occurrences of this abscissa of convergence in the literature.

For $\Re(s)>\log_b(N)$ the partial series are thus absolutely convergent and
the formula \eqref{eq:startswithw} is valid.  Throughout the paper from now
on, we assume $\Re(s)>\log_b(N)$.

We appeal to the binomial series and intervert series and integral in \eqref{eq:startswithw}:
\begin{equation}\label{eq:binomial}
  \int_{z\in\cW} \sum_{m=0}^\infty (-1)^m\frac{(s)_m}{m!}\frac{x(z)^m}{n(w)^{m+s}}\dmu_{A,s}(z)
=\sum_{m=0}^\infty (-1)^m\frac{(s)_m}{m!}\frac{u_{A,s}(m)}{n(w)^{m+s}}\;.
\end{equation}
Here, $(s)_m = s(s+1)\dots (s+m-1)$ is the ascending factorial with $m$ terms
(Pochhammer symbol).  We have assumed $n(w)>1$ so that the series depending on
$x(z)$ is uniformly and absolutely convergent on $\cW$ and termwise
integration is justified.  It is asserted in the next proposition that this
interversion is valid also for $n(w)=1$.
\begin{proposition}\label{prop:conv}
  Let $s$ be a complex number
  such that $\Re s >\log_b(N)$.  The series
  \begin{equation*}
    L_{b,A}(s) \coloneq \sum_{m=0}^\infty(-1)^m\frac{(s)_m}{m!}u_{A,s}(m),
  \end{equation*}
  converges and is equal to $\int_{\cW}\bigl(1+x(w)\bigr)^{-s}\dmu_{A,s}(w)$.
  Let
  $\lambda = \frac{\max A}{b-1}\in \Iof01$.
  \begin{itemize}
  \item The remainders are $o(\lambda^{m})$. If $\Re s\geq1$, they are more precisely
    $O(\lambda^m/m)$.
  \item  If $\lambda=1$ and $s$ is real, the series
    is only semi-convergent.
  \end{itemize}
  Multiplying each term with an additional factor $n^{-m-s}$ for $n$ some
  integer $>1$ (or any real number $>1$), we thus obtain a geometrically
  convergent series.  For $s$ real, its partial sums are upper and lower
  bounds depending on parity of number of terms.
\end{proposition}
\begin{proof}
  In this proof, we appeal to estimates which will be established only later.
  Let us start first with a Taylor expansion with integral remainder for the
  function $(n+x)^{-s}$ where $n>0$ is for now some arbitrary positive real
  number, $s$ is a complex number, and $x\geq 0$.  It reads:
\begin{equation*}
  (n + x)^{-s} - \sum_{k=0}^{m-1} (-1)^k\frac{(s)_k}{k!n^s} \frac{x^k}{n^k}
  = (-1)^m\frac{(s)_m}{(m-1)!}\int_{0}^x (x-t)^{m-1}(n+t)^{-s-m}\dt.
\end{equation*}
We observe in passing that, if $s$ is a positive real number, the sign of the
left-hand side has thus to be $(-1)^m$, except of course for $x=0$. Let now $s$
again be some arbitrary complex number, with real part $\sigma=\Re s$.  We get:
\begin{equation*}
  \left| (n + x)^{-s} - \sum_{k=0}^{m-1} (-1)^k\frac{(s)_k}{k!n^s} \frac{x^k}{n^k} \right|
  \leq  \frac{|(s)_m|}{(m-1)!n^{\sigma+m}}\int_{0}^x (x-t)^{m-1}(1+n^{-1}t)^{-\sigma-m}\dt.
\end{equation*}
Suppose that $\sigma+m\geq0$.  Then (and this is surely a standard calculus exercise):
\begin{equation*}
  \left| (n + x)^{-s} - \sum_{k=0}^{m-1} (-1)^k\frac{(s)_k}{k!n^s} \frac{x^k}{n^k} \right|
  \leq  \frac{|(s)_m|x^m}{m! n^{\sigma+m}}\;.
\end{equation*}
We assume from now on that $\sigma=\Re s > \log_b N$. So in particular the
hypothesis $\sigma+m\geq0$ is definitely true, even $\sigma+m>0$. We
apply the above upper bound after replacing $x$ by the function $x(w)$ on
$\cW$. For now, we still do not suppose $n$ to be an integer, only $n>0$.
\begin{align*}
\textrm{\llap{Let }}  \delta_m(s,n) &\coloneq
 \left| \int_{z\in\cW} \bigl(n + x(z)\bigr)^{-s}\dmu_{A,s}(z)
 - \sum_{k=0}^{m-1} (-1)^k\frac{(s)_k}{k!n^{s+k}} u_{A,s}(k)
   \right|
\\
&=
 \left| \int_{z\in\cW} \biggl(\bigl(n + x(z)\bigr)^{-s}
 - \sum_{k=0}^{m-1} (-1)^k\frac{(s)_k}{k!n^{s+k}} x(z)^k\biggr)
 \dmu_{A,s}(z)  \right|
\\
&\leq
 \int_{z\in\cW} \left| \bigl(n + x(z)\bigr)^{-s}
 - \sum_{k=0}^{m-1} (-1)^k\frac{(s)_k}{k!n^{s+k}} x(z)^k
 \right| \dmu_{A,\sigma}(z)
\\
&\leq
 \int_{z\in\cW} \frac{|(s)_m|x(z)^m}{m!n^{\sigma + m}}\dmu_{A,\sigma}(z)
=\frac{|(s)_m|}{m!n^{\sigma + m}}u_{A,\sigma}(m).
\end{align*}
Recall that, for any complex $z$, there holds $(z)_k/k!\sim_{k\to\infty}
\Gamma(z)^{-1}k^{z-1}$ (with $0\sim 0$ if $z\in-\NN$).
Thus, for $\sigma=\Re z\neq 0,-1,-2, \dots$, the
sequence $k\mapsto \frac{|(z)_k|}{|(\sigma)_k|}$, which is non-decreasing, has
limit equal to $\frac{|\Gamma(\sigma)|}{|\Gamma(z)|}$, hence, it is bounded
above by this limiting value.  In particular, for our $s$ (which has
$\sigma=\Re s >0$), and any real number $n>0$, we have the upper bound:
\begin{equation*}
  \delta_m(s,n) \leq
\frac{|(s)_m|}{(\sigma)_m}\frac{(\sigma)_mu_{A,\sigma}(m)}{m!n^{\sigma+m}}
\leq
  \frac{\Gamma(\sigma)}{|\Gamma(s)|}
  \frac{(\sigma)_m}{m!}\frac{u_{A,\sigma}(m)}{n^{\sigma+m}}
  \sim_{m\to\infty}
  \frac{m^{\sigma-1}}{|\Gamma(s)|}\frac{u_{A,\sigma}(m)}{n^{\sigma+m}}\;.
\end{equation*}
Further, we know that $u_{A,\sigma}(m)$ is $O(1)$, even that it decreases to
zero.  Thus, the quantity $\delta_m(s,n)$ is $o(m^{\sigma-1}n^{-m})$
(we will obtain a better bound for $\sigma\geq1$ below).

For $n>1$, we thus have $\delta_m(s,n)=o(c^m)$ for any $c>n^{-1}$. This gives
a more detailed justification for the step from \eqref{eq:startswithw} to
\eqref{eq:binomial}.

Let us now assume $n=1$.  The quantity $\delta_m(s,1)$ is the distance from
the $m$-th partial sum of $L_{b,A}(s)$ (i.e. keeping terms with indices $<m$)
to the integral $\int_{z\in\cW} \bigl(1 + x(z)\bigr)^{-s}\dmu_{A,s}(z)$.  We
have obtained
$
  \delta_m(s,1)\leq   \frac{\Gamma(\sigma)}{|\Gamma(s)|}
  \frac{(\sigma)_m}{m!} u_{A,\sigma}(m)
$.

Suppose first that $\sigma<1$. Thus $\delta_m(s,1)=o(1)$. More precisely, from
next Lemma \ref{lem:geo},
$\delta_m(s,1)=o(m^{\sigma-1}\lambda^m)=o(\lambda^m)$ where $\lambda=\max
A/(b-1)$.  This establishes the convergence and the value of $L_{b,A}(s)$.

Take now $\sigma\geq1$.  From Proposition \ref{prop:upper}, which is stated
later, we deduce that $\delta_m(s,1)$ is $O(\lambda^m m^{-1})$ and
consequently is $o(\lambda^m)$.

Finally, assuming $s=\sigma$ is real, we have from Proposition
\ref{prop:lower}, also stated later, the lower bound
$\frac{(\sigma)_m}{m!}u_{A,\sigma}(m)>
 \frac{\sigma}{\sigma+m}\lambda^m/(b^\sigma-N)$.
Hence, if $\lambda=1$ the series defining $L_{b,A}(\sigma)$ is only
semi-convergent.
\end{proof}
We shall now establish the various estimates used in the
previous proof.  We use the notation $f = \max A$. So $f \leq b-1$.
First, a result applying also to complex $s$: 
\begin{lemma}\label{lem:geo}
  For any $s$ with $\Re s>\log_b(N)$,
  the sequence $(\Bigl(\frac{b-1}{f}\Bigr)^m u_{A,s}(m))_{m\geq0}$ has limit zero.
\end{lemma}
\begin{proof}
  Due to $|u_{A,s}(m)|\leq u_{A,\Re s}(m)$, we can assume $s$ to be real.  Define
  the rescaled function $y(w)= \frac{b-1}{f} x(w)$.
  Thus, $0\leq y(w)< 1$ for every \emph{admissible} $w\in \cW$. The
  non-admissible $w$ may have $y(w)\geq1$ but they live in a set of zero
  measure.  By the Lebesgue dominated convergence theorem, the integral against
  $\mu_{A,s}$ of the powers of the function $y(w)$ go to zero (and as here $s$ is real,
  we can even say that the rescaled moments decrease to zero).
\end{proof}

Now for lower and upper bounds with $s=\sigma$ real.
\begin{proposition}\label{prop:lower}
  Let $m\geq0$.
  For $\sigma>\log_b(N)$ there holds (with $f=\max A$):
  \begin{equation*}
    u_{A,\sigma}(m)> \frac \sigma{m+\sigma}\frac{m!}{(\sigma)_m}\frac{(f/(b-1))^m}{b^\sigma - N}\;.
  \end{equation*}
\end{proposition}
\begin{proof}
  Let us temporarily set $v_{A,\sigma}(m) = u_{A,\sigma}(m)/u_{A,\sigma}(0)$.  They obey the
  same recurrence \eqref{eq:recu} as the $u_{A,\sigma}(m)$:
\begin{equation*}
  (b^{m+\sigma} - N) v_{A,\sigma}(m) = \sum_{j=1}^{m} \binom{m}{j}(\sum_{a\in A}a^j) v_{A,\sigma}(m-j).
\end{equation*}
We can only decrease these quantities if we replace in this recurrence
$b^{m+\sigma}-N$ by its upper bound $b^{m+\sigma}-1$ in the left-hand side, and
$\sum_{a\in A} a^j$ by its lower bound $f^j$, where $f = \max A$.  So
$v_{A,\sigma}(m)\geq w_{A,\sigma}(m)$ where $w_{A,\sigma}(0) = 1$ and for $m\geq1$:
\begin{equation*}
  (b^{m+\sigma} - 1) w_{A,\sigma}(m) = \sum_{j=1}^{m} \binom{m}{j}f^j w_{A,\sigma}(m-j).
\end{equation*}
But this means that $w_{A,\sigma}(m) = u_{A',\sigma}(m)/u_{A',\sigma}(0)$ where $A'$ is the
singleton $\{f\}$. For the $A'$-admissible words $f^j$, one has $x(f^j)=
(f/(b-1))(1 - b^{-j})$. So:
\begin{equation*}
  u_{A',\sigma}(m) = f^m(b-1)^{-m}\sum_{j=0}^\infty b^{-\sigma j}(1 - b^{-j})^m.
\end{equation*}
Now, consider the function $t\mapsto g(t) = (1 - t^{\sigma^{-1}})^m$ on $\Iff01$,
and assume $m>0$.  This function is strictly decreasing.  Considering the
evaluation points $\xi_j=b^{-\sigma j}$, $j\geq1$, we obtain:
\begin{equation*}
  \int_0^1  (1 - t^{\sigma^{-1}})^m\dt < g(\xi_1)(1-\xi_1) + g(\xi_2)(\xi_1 - \xi_2) + \dots
  = \sum_{j=1}^\infty (1 - b^{-j})^mb^{-\sigma j}(b^{\sigma} - 1).
\end{equation*}
So, putting everything together, we obtain, for $\sigma>\log_b(N)\geq0$, $m\geq1$:
\begin{equation*}
  u_{A,\sigma}(m)> \frac{b^\sigma}{b^\sigma -N}\frac{b^\sigma -1}{b^\sigma}
                  \frac {f^m}{(b-1)^m}\frac{\int_0^1 (1 - t^{\sigma^{-1}})^m\dt}{b^\sigma -1}
  = \frac \sigma{m+\sigma}\frac{m!}{(\sigma)_m}\frac{(f/(b-1))^m}{b^\sigma - N}\;.
\end{equation*}
We used the well-known formula for the eulerian beta integral. This final result holds
also for $m=0$ and completes the proof of the lower bound.
\end{proof}

In view of the previous inequality we now define, for $m\geq0$ (and complex $s$):
\begin{equation}\label{eq:defu*}
   u_{A,s}^*(m) =\frac{(s+1)_m}{m!} u_{A,s}(m) = \frac{s+m}{s}\frac{(s)_m}{m!} u_{A,s}(m),
\end{equation}
where again $(s+1)_m$ is an \emph{ascending} partial factorial.
With these quantities, \eqref{eq:recu} becomes
\begin{equation}\label{eq:recu*}
 (b^{m+s} -N) u_{A,s}^*(m) = \sum_{j=1}^{m} \frac{(s+m)\dots(s+m-j+1)}{j!}
                                          (\sum_{a\in A}a^j) u_{A,s}^*(m-j).
\end{equation}

For the upper bound we will need the following lemma from calculus:
\begin{lemma}\label{lem:calculus}
  Let $\sigma\geq1$ be a real number, $m$ a non-negative integer, and $a$ a
  non-negative real number.  Then
\begin{equation*}
  \sum_{j=1}^m \frac{(\sigma+m)\dots(\sigma+m-j+1)}{j!} a^j \leq (1+ a)^{\sigma+m}- 1 - a^{\sigma+m}\,.
\end{equation*}
\end{lemma}
\begin{proof}
  Of course the left-hand side is defined to be zero if $m=0$. The proof can
  be done by induction on $m$, or, one can apply the Taylor formula with
  integral remainder for the expansion of $(1+a)^{\sigma+m}$ in ascending powers of
  $a$ to $m$-th order (i.e.\@ with the remainder involving the $(m+1)$-th
  derivative). This expansion gives $1$, plus the above sum on left-hand side, plus:
  \begin{equation*}
    \int_0^a \frac{(a-t)^m}{m!}(\sigma+m)\dots (\sigma+1)\sigma(1 + t)^{\sigma-1}\dt.
  \end{equation*}
  As $\sigma\geq1$, we can use $(1+t)^{\sigma-1}\geq t^{\sigma-1}$ in this integral.  Then we
  conclude via a Beta integral formula or we recognize that we have now the
  remainder integral for the evaluation of $(0 + a)^{\sigma+m}$ in ascending powers
  of $a$ up to order $m$.  So this lower bound gives $a^{\sigma+m}$ and moving
  $a^{\sigma+m}$ and $1$ to the other other side of the inequality we obtain the
  result.
\end{proof}
\begin{proposition}\label{prop:upper}
  Let $\sigma\geq1$ ($\sigma>1$ if $N=b$). Let $f=\max A$. There holds, for every $m\geq0$:
  \begin{equation*}
    \frac {\sigma+m}\sigma\frac{(\sigma)_m}{m!} u_{A,\sigma}(m) 
   = u_{A,\sigma}^*(m)\leq  \frac{f^m}{(b-1)^m}\frac{b^\sigma}{b^\sigma - N}\;.
  \end{equation*}
\end{proposition}
\begin{proof}
  Let $c=u_{A,\sigma}^*(0)$.  Let $m\geq1$ and suppose by induction hypothesis
  $u_{A,\sigma}^*(m-j)\leq (f/(b-1))^{m-j}c$ for $1\leq j\leq m$.  The power sums
  $\sum_{a\in A} a^j$ with $j\geq 1$ are maximal when the $N$ admissible
  digits are $f$, $f-1$, \dots, $f-N+1$. In fact, we can stop with $f-N_1+1$,
  because if $N_1<N$, this means that $A$ contains the digit $0$, and the
  latter does not contribute to the power sums anyhow.  So, from
  \eqref{eq:recu*}:
\begin{align*}
  (b^{m+\sigma} -N) &u_{A,\sigma}^*(m) \leq c \sum_{j=1}^{m}
  \frac{(\sigma+m)\dots(\sigma+m-j+1)}{j!}
  \Bigl(\frac{f}{b-1}\Bigr)^{m-j}\sum_{f-N_1< a \leq f} a^j
  \\
&\leq  \Bigl(\frac{f}{b-1}\Bigr)^m c\sum_{j=1}^{m} \frac{(\sigma+m)\dots(\sigma+m-j+1)}{j!}
    \sum_{f-N_1< a \leq f} \Bigl(\frac{(b-1)a}{f}\Bigr)^j\,.
\end{align*}
By hypothesis, $\sigma\geq1$, so, using Lemma \ref{lem:calculus}, we obtain:
\begin{equation}\label{eq:foo}
  \begin{aligned}[b]
    (b^{m+\sigma} -N) &u_{A,\sigma}^*(m) 
    \\
    &\leq \Bigl(\frac{f}{b-1}\Bigr)^m
              c\sum_{f-N_1< a \leq f}
    \Bigl((\frac{b-1}{f}a+1)^{\sigma+m} - 1 - \bigl(\frac{b-1}{f}a\bigr)^{\sigma+m} \Bigr).
  \end{aligned}
\end{equation}
As the
smallest used $a$ is at least $1$, there holds:
\begin{equation*}
\frac{b-1}{f}a \geq \frac{b-1}{f}(a-1) + 1 \geq 1.
\end{equation*}
We can thus, as $\sigma+m\geq0$ (even $\sigma+m\geq2$), replace \eqref{eq:foo} by a
weaker upper bound which has the advantage of being telescopic. After summing
the contributions, we replace the last term $-(\frac{b-1}{f}(f-N_1)+1)^{\sigma+m}$
by its upper bound $-1$.  Hence,
\begin{equation*}
  (b^{m+\sigma} -N) u_{A,\sigma}^*(m)
\leq \Bigl(\frac{f}{b-1}\Bigr)^m \bigl(b^{\sigma+m} - N_1 - 1\bigr)c
\leq  \Bigl(\frac{f}{b-1}\Bigr)^m \bigl(b^{\sigma+m} - N\bigr)c.
\end{equation*}
Hence, the validity for $u_{A,\sigma}^*(m)$ of the induction hypothesis is
established and the proof is complete.
\end{proof}

Having by now completely justified the Proposition \ref{prop:conv} we return
to the discusion, initiated with equations \eqref{eq:startswithw} and
\eqref{eq:binomial}, of the restricted Dirichlet series.
\begin{theorem}\label{thm:main}
  Let $\ell\geq1$. There holds:
  \begin{equation*}
    \begin{aligned}[t]
      K_{b,A,s} = &\sideset{}{'_{0<n<b^{\ell-1}}}\sum \frac1{n^s}
      +  \frac{b^{s}}{b^{s} -N}\sideset{}{'_{b^{\ell-1}\leq n<b^{\ell}}}\sum \frac1{n^s}\\
      &+ \sum_{m=1}^\infty (-1)^m \frac{(s)_{m}}{m!}u_{A,s}(m)
              \sideset{}{'_{b^{\ell-1}\leq n<b^{\ell}}}\sum \frac1{n^{m+s}}\;.
    \end{aligned}
  \end{equation*}
  The partial sums provide, for $s$ real, lower and upper bounds depending on
  the parity of number of terms. For complex $s$, if either $\ell>1$, or
  $\ell=1$ and $1\notin A$, or $\ell=1$ and $b-1\notin A$, the series is
  geometrically convergent.
\end{theorem}
\begin{proof}
  It is only a matter to point out first that for any positive integer $n$ (or
  even any real number $\geq1$) one has:
\begin{equation*}
  \sum_{\substack{m \geq 0\\\text{admissible}}}\frac1{\bigl(n b^{\ell(m)} +m\bigr)^s}
=
\int_{z\in\cW} \bigl(n + x(z)\bigr)^{-s}\dmu_{A,s}(z)
=
\sum_{m=0}^\infty (-1)^m\frac{(s)_m}{m!}\frac{u_{A,s}(m)}{n^{m+s}}\;,
\end{equation*}
where the first equality is justified as was done for \eqref{eq:startswithw}
where $n$ was an admissible integer (that step only needs $n$ to be a positive
real number), and the second equality is easy if $n>1$ from uniform and
absolute convergence of the binomial series, whereas for $n=1$ it is the main fact
stated by Proposition \ref{prop:conv}.

Then for any level $\ell\geq1$, we gather in $K_{b,A,s}$ the subseries
contributed by those admissible integers ``starting with'' the finitely many
admissible integers in $\Iffint{b^{\ell-1}}{b^\ell-1}$.  The properties of the
individual series stated in Proposition \ref{prop:conv} (the only delicate case
being with $\ell=1$, $n=1$, $1\in A$, $b-1\in A$) complete the proof of the
theorem.
\end{proof}
\begin{remark}
  For numerical computations one should anyhow use at least $\ell=2$, except
  perhaps if $b$ is large; and one can, from $(1+x)^{-s} = (2 - (1-x))^{-s}$
  and expansion in powers of $1-x$, hence using ``complementary moments'' as
  in \cite{burnolirwin}, always provide a series with geometric convergence
  also for the partial $K$-series of those $n^{-s}$ with $n$ ``starting with''
  $1$, or more generally (if $1\notin A$) the $L_{b,A}(s)$ series of
  Proposition \ref{prop:conv}.
\end{remark}

\begin{remark}
  With the quantities $u_{A,s}^*(m)$ (of equation \eqref{eq:defu*}) the formula
  of Theorem \ref{thm:main} for $K_{b, A, s}$ becomes:
  \begin{equation}\label{eq:Ku*}
    \begin{split}
      K_{b,A,s} ={} &\sideset{}{'_{0<n<b^{\ell-1}}}\sum \frac1{n^s}
      +  \frac{b^{s}}{b^{s} -N}\sideset{}{'_{b^{\ell-1}\leq n<b^{\ell}}}\sum \frac1{n^s}\\
      &+ \sum_{m=1}^\infty (-1)^m u_{A,s}^*(m)\frac s{s+m}
      (\sideset{}{'_{b^{\ell-1}\leq n<b^{\ell}}}\sum \frac1{n^{s+m}}).
    \end{split}
  \end{equation}
\end{remark}

\section{Moment generating function}

Let again $b>1$ an integer, $\sD=\Iffint 0{b-1}$, $\cW$ the set of words on
the alphabet $\sD$ and $A\subset \sD$, $N=\#A$, $s\in \CC$, $\Re s>
\log_b(N)$.  Recall from Definition \ref{def:1} the map
$x:\cW\to\Ifo01$.
\begin{definition}
  The (exponential) moment generating function $E_{A,s}$ is the entire
  function defined by the formula:
  \begin{equation*}
    E_{A,s}(t) = \int_{\cW} \e^{t x(w)}\dmu_{A,s}(w).
  \end{equation*}
\end{definition}
\begin{proposition}\label{prop:Et}
  Let $\alpha_A(t)=\sum_{a\in A} e^{at}$.
  The following holds:
  \begin{equation}\label{eq:Et}
    \forall t\in \CC\quad E_{A,s}(t) = 1 + \sum_{j\geq1}b^{-js}\prod_{i=1}^j\alpha_A(b^{-i}t).
  \end{equation}
  There is a functional equation:
  \begin{equation}\label{eq:funceq}
    E_{A,s}(t) = 1 + b^{-s}\alpha_A(b^{-1}t) E_{A,s}(b^{-1}t).
  \end{equation}
\end{proposition}
\begin{proof}
  Indeed for $j\geq1$,
  \begin{equation*}
    \prod_{i=1}^j\alpha_A(b^{-i}t) = \sum_{a_1,\dots, a_j\in A}\e^{(a_1b^{-1} + \dots + a_j b^{-j})t}
     = \mathop{\sum\nolimits'}\limits_{\ell(w)=j}\e^{x(w)t}\,.
  \end{equation*}
   This gives \eqref{eq:Et}.  And the functional equation \eqref{eq:funceq} follows.
\end{proof}
In the series from \eqref{eq:Et},
$\bigl|\prod_{i=1}^j\alpha_A(b^{-i}t)\bigr|\leq N^j \e^{\max(0,\Re t)}$, as we
see from the above expansion of the product over the $N^j$ admissible words of
length $j$. One can also reformulate equation \eqref{eq:Et} to read:
\begin{equation}
  \label{eq:Et2}
   E_{A,s}(t) = 1 + \sum_{j\geq1}\frac{1}{(b^s/N)^{j}}\prod_{i=1}^j\frac{\alpha_A(b^{-i}t)}{N}\;.
\end{equation}
Indeed, as $\alpha_A(z)/N = 1 + O_{z\to0}(z)$, the infinite product
$\beta_A(t) \coloneq \prod_{i=1}^\infty \frac{\alpha_A(b^{-i}t)}N$ is convergent,
everywhere in the complex plane (it vanishes at the $b^jz_0$, $j\geq1$,
$\alpha_A(z_0)=0$, which exist except if $A$ is a singleton).  So, either the
series \eqref{eq:Et2} has only finitely many non-zero terms (if $t$ is a zero
of $\beta_A$) or its terms are equivalent up to some multiplicative factor
to those of the geometric series of ratio $N/b^s$.  And we are assuming
$\Re(s)>\log_b(N)$ so this latter series converges.

Suppose (provisorily) that $0\in A$, then for
$\Re t<0$ the infinite product $\gamma_A(t)= \prod_{i=0}^\infty
\alpha_A(b^{i}t)$ is convergent.  Expanding in powers of
$q=\e^t$, $|q|<1$, we recognize the series $\sideset{}{'_{n\geq0}}\sum
q^n$ where the exponents are the \emph{admissible} integers.

One checks that the product $G_A(t)=\beta_A(t) \gamma_A(t)$ (defined for $\Re
t<0$) verifies the functional equation $G_A(bt) = N^{-1}G_A(t)$.  However, we
prefer to not pursue that discussion, in favour of another direction, very
much alike in terms of tools, which does not make the assumption that $0$ is
an admissible digit and is more focused on properties of $E(t)$ for
$t\to+\infty$.

For this, let $f=\max A$, $B = f - A$ so that it contains the zero digit, and
let us write $\sum_{a\in A} \e^{at}$ as $\e^{ft}\alpha_B(-t) =
\e^{ft}\sum_{a'\in B}\e^{-a't}$.  The formula for the moment generating
function becomes:
\begin{equation}
  \label{eq:Et3}
   E_{A,s}(t) =
   1 + \sum_{j\geq1}b^{-js}\e^{f(b-1)^{-1}(1 - b^{-j})t}\prod_{i=1}^j\alpha_B(-b^{-i}t).
\end{equation}
We observe that for $\Re t>0$, defining
$\gamma_B(-t)\coloneq\prod_{i=0}^\infty \alpha_B(-b^i t)$, we obtain a
convergent infinite product, which in the variable $q=\e^{-t}$ is
the subseries of $\sum q^n$ where the exponents are restricted to have their
radix $b$ digits only from the set $B= f-A$ (in particular $n=0$ is there).
Let $\lambda = f/(b-1)\in \Iof01$.  We observe that we can reformulate
\eqref{eq:Et3}, for $\Re t>0$, as:
\begin{equation}
  \label{eq:Et4}
   t^s\e^{-\lambda t}\gamma_B(-t)E_{A,s}(t) =
    \sum_{j=0}^\infty \bigl(b^{-j}t\bigr)^s\e^{ - \lambda b^{-j}t}\gamma_B(-b^{-j}t).  
\end{equation}
We realize that negative $j$'s would give a very quickly convergent series.
Let us exchange $j$ with $-j$ and define for $\Re t>0$
\begin{equation}
  \label{eq:F}
  F_{A,s}(t) = 
  \sum_{j=-\infty}^\infty \bigl(b^{j}t\bigr)^s\e^{-\lambda b^{j}t}\gamma_B(-b^{j}t).
\end{equation}
The ``corrections'' (now associated with positive $j$'s, the original series
being with non-positive $j$'s) go to zero exponentially fast when $t\to+\infty$.
The limit function being invariant under $t\mapsto bt$, it is susceptible to
be expanded, say for $t>0$, into a Fourier series in the variable $\log t$ (or
better, a Laurent series in the variable $t^{2\pi i/\log b}$, but we do not go
into this here).
\begin{proposition}\label{prop:Ft}
  Let $f=\max A$, $\lambda = \frac{f}{b-1}$, $B$ the set $f-A$ and $\cB$ the
  set of $B$-admissible (non-negative) integers. Define
  $\phi_B(q)=\sum_{n\in\cB} q^n$ for $|q|<1$.  For any $t$ with positive real
  part there exists
\[
\lim_{k\to+\infty} (b^k t)^s \e^{-b^k\lambda t}E_{A,s}(b^k t),
\]
and its value is, using the notation $q=\e^{-t}$:
\[
   t^s\e^{-\lambda t}\phi_B(q)E_{A,s}(t) + \sum_{j=1}^\infty
    \bigl(b^{j}t\bigr)^s\e^{-\lambda b^{j}t}\phi_B(q^{b^j}).
\]
This limit $F_{A,s}(t)$ is an analytic function of $t$ on the half-plane $\Re
t>0$ and it is invariant under $t\mapsto bt$.  Restricting to real positive
$t$'s, its Fourier coefficients as a $1$-periodic function of $\log_b t$ are:
\[
\int_0^1 F_{A,s}(b^u)e^{-2\pi i k u}\du 
= \frac{\Gamma(s - 2\pi i \frac k{\log b})}{\log b}
  \sum_{n\in\cB}\frac{1}{(n + \lambda)^{s-2\pi i \frac k{\log b}}}\;.
\]
\end{proposition}
\begin{proof}
  In equation \eqref{eq:F}, the non-positive indices contribute
  $t^s\e^{-\lambda t}\phi_B(q)E_{A,s}(t)$, as shown in equation \eqref{eq:Et4}.
  Replacing in \eqref{eq:F} $t$ by some $b^kt$, the value of $F_{A,s}$ remains
  invariant, but the total contribution of positive indices clearly goes to zero as
  $k\to+\infty$.  We have $\phi_B(q^{b^k}) = \gamma_B(-b^k t) = 1 +
  O_{k\to\infty}(\e^{-b^kt})$, so we can remove that term when considering the
  limit of $ (b^k t)^s \e^{-b^k\lambda t}\phi_B(q^{b^k})E_{A,s}(b^k t)$.  This
  gives the existence and value of the limit from the proposition.

  We compute the
  average of $F_{A,s}$ over a period.
\begin{equation*}
  \begin{split}
    \int_0^1 F_{A,s}(b^u)\du
&= \int_{-\infty}^\infty b^{su}\e^{-\lambda b^u}\phi_B(\e^{-b^u})\du 
\\
&= \int_0^\infty t^{s}\e^{-\lambda t}\sum_{n\in\cB} e^{-nt}\frac{\dt}{t\log b}
= \frac{\Gamma(s)}{\log b}\sum_{n\in\cB}\frac{1}{(n + \lambda)^s}\;.
  \end{split}
\end{equation*}
There are no problems of convergence (everything converges absolutely): $B$
has the same cardinality $N$ as $A$, and $\Re s > \log_b N$.  The general
formula for the $k$-th Fourier coefficient is obtained the same way.
\end{proof}

Doubly infinite sums over powers of $b$, similarly to what happens in
\eqref{eq:F}, creating periodicity, the computation of Fourier coefficients,
the contrasting behaviour of the contributions from positive versus negative
indices, all of that has been encountered in various contexts in the literature
and we refer to \cite{hardy1907,mendsebb1999,keatread2000,
  balamendsebb2005,zhang2011,zhang2023} for further considerations.

\section{The case of the zeta function}

Let $b>1$ an integer.  Choosing $A=\sD=\Iffint{0}{b-1}$, $K_{b,A,s}$ is
nothing else than the Riemann zeta function $\zeta(s)$, $\Re s>1$, and after
choosing a level $l\geq2$, the Theorem \ref{thm:main} gives geometrically
converging series representing $\zeta(s)$, in terms of some ``moments''
obeying the recurrence relations \eqref{eq:recu} or \eqref{eq:recu*}.  This
recurrence relation is also embedded in the properties of the
moment-generating function $E_{A,s}(t)$ from the previous section.  As here
$A$ is actually determined by $b$, we change the notation into $E_{b,s}$, and
similarly the function $\alpha_A$ of Proposition \ref{prop:Et} is now:
\begin{equation*}
  \alpha_b(t) = \sum_{a=0}^{b-1}\e^{at} = \frac{\e^{bt}-1}{\e^t - 1}\;.
\end{equation*}
We will also switch from $u_{A,s}(m)$ to $u_{b,m}(s)$ for the moments
(from definition \ref{def:um}) or even $u_m(s)$. So, equation \eqref{eq:Et} becomes
\begin{equation}
  E_{b,s}(t) = \sum_{j=0}^\infty b^{-js}\frac{\e^{t}-1}{\e^{t/b^{j}}-1}
         = \frac{e^t - 1}t
           \sum_{j=0}^\infty b^{-j(s-1)}\frac{t/b^{j}}{\e^{t/b^{j}}-1}\;.
\end{equation}
The $F_{A,s}(t)$ from Proposition \ref{prop:Ft}, now denoted $F_{b,s}(t)$, is:
\begin{equation*}
  F_{b,s}(t)= \frac{t^s}{\e^t -1}E_{b,s}(t) + \sum_{j=1}^\infty
    \frac{\bigl(b^{j}t\bigr)^s}{\e^{b^j t}-1} 
   = \sum_{j=-\infty}^\infty \frac{\bigl(b^{j}t\bigr)^s}{\e^{b^j t}-1} \;.
\end{equation*}
We take note that for $t>0$, it is a discretization of the representation of
$\Gamma(s)\zeta(s)$ as an integral over $(0,\infty)$:
\begin{equation*}
  \Gamma(s)\zeta(s) = \int_0^\infty \frac{u^{s-1}}{\e^u - 1}\du 
  = \int_{-\infty}^\infty \frac{\e^{sx}}{\exp(\e^x)-1}\dx \approx
  \log b \sum_{j=-\infty}^\infty \frac{\bigl(b^{j}t\bigr)^s}{\e^{b^j t}-1}\;,
\end{equation*}
where the approximation becomes exact in the limit $b\to 1^+$.  This
$\Gamma(s)\zeta(s)/\log b$ is the zeroth Fourier coefficient mentioned in
Proposition \ref{prop:Ft}.

We can express the moment generating function
using the Bernoulli numbers:
\begin{align*}
  \sum_{j=0}^\infty b^{-j(s-1)}\frac{t/b^{j}}{\e^{t/b^{j}}-1}
&=
\sum_{j=0}^\infty b^{-j(s-1)} 
\begin{aligned}[t]
  &-\frac t2\sum_{j=0}^\infty b^{-j(s-1)}b^{-j}\\
  &+ \sum_{k=1}^\infty \frac{B_{2k}t^{2k}}{(2k)!}\sum_{j=0}^\infty b^{-j(s-1)}b^{-2kj}
\end{aligned}
\\
&=
\frac{b^{s}}{b^s -b} -\frac t2\frac{b^s}{b^s-1}
+ \sum_{k=1}^\infty \frac{B_{2k}t^{2k}}{(2k)!}\frac{b^{s+2k}}{b^{s+2k}-b}\;.
\end{align*}
Consequently,
\begin{equation}
\label{eq:uber}
\begin{split}
  \sum_{m=0}^\infty \frac{u_m(s)}{m!}t^m = \frac{b^{s}}{b^s-b}\frac{\e^t-1}{t}
              &- \frac{b^s}{2(b^s-1)}(\e^t-1)
  \\
              &+\frac{\e^t-1}{t}\sum_{k=1}^\infty 
                               \frac{B_{2k}t^{2k}}{(2k)!}\frac{b^{s+2k}}{b^{s+2k}-b}\;,
\end{split}
\end{equation}
and, in particular,
\begin{align*}
u_{0}(s) &= \frac{b^s}{b^s-b},\\
u_1(s) &= \frac{b^s(b-1)}{2(b^s-b)(b^s-1)},
\\
u_2(s) &= \frac13\frac{b^s}{b^s-b} - \frac12\frac{b^{s+1}}{b^{s+1}-b} 
        + \frac1{6}\frac{b^{s+2}}{b^{s+2} - b}\;.
\end{align*}
One checks that the given values for the first three moments are compatible
with the general formulas \eqref{eq:u0}, \eqref{eq:u1}, and \eqref{eq:u2}
respectively.

From Theorem \ref{thm:main} (and equations \eqref{eq:defu*},
\eqref{eq:recu*}, \eqref{eq:Ku*}), Propositions \ref{prop:bound}, \ref{prop:lower} and
\ref{prop:upper}, and the above equation \eqref{eq:uber}, we obtain the
concluding statement of this paper:
\begin{theorem}\label{thm:conclusion}
  Let $b>1$ be an integer and $s$ a complex number of real part $\sigma$
  greater than $1$.
  Let
  \begin{align*}
  u_0(s) &= \frac{b^s}{b^s -b}\,,\\
\shortintertext{and, for $m\geq1$:}
  u_m(s) &=  \frac1{m+1}\frac{b^{s}}{b^s -b} - \frac{b^{s+1}}{2(b^{s+1}-b)} +
\sum_{1\leq k \leq \lfloor \frac m2 \rfloor}
     \frac{m!}{(m-2k+1)!}\frac{B_{2k}}{(2k)!}\frac{b^{s+2k}}{b^{s+2k}-b}\;.
  \end{align*}
There holds $|(b^s-b)u_m(s)|\leq (b^\sigma-b)u_m(\sigma)$ for $m\geq0$.\newline
Let, for $m\geq0$, $u_m^*(s) = \frac{(s+1)_m}{m!}u_m(s)$.
They obey the recurrence, for $m\geq1$:
\begin{equation}\label{eq:recum*}
  u_{m}^*(s) = \frac1{b^{m+s} -b}
  \sum_{j=1}^{m} \frac{(s+m)\dots(s+m-j+1)}{j!}u_{m-j}^*(s)(\sum_{0\leq a<b}a^j),
\end{equation}
and verify, for $s=\sigma>1$ real, the uniform estimates:
\begin{equation*}
  \frac{1}{b^\sigma - b} < u_{m}^*(\sigma) \leq \frac{b^\sigma}{b^\sigma- b}\;.
\end{equation*}
For general $s$, $\Re s = \sigma>1$:
\begin{equation*}
  |u_{m}^*(s)| \leq
  \frac{|(s+1)_{m}|}{(\sigma +1)_{m}}\frac{b^\sigma}{|b^s - b|}
  \leq \frac{\Gamma(\sigma +1)}{|\Gamma(s+1)|}\frac{b^\sigma}{|b^s- b|}\;.
\end{equation*}
For any integer $\ell\geq2$, one has the absolutely convergent series representation:
\begin{equation}\label{eq:zetaseries}
        \zeta(s)= \sum_{0<n<b^{\ell-1}} \frac1{n^s}
      + \sum_{m=0}^\infty (-1)^m \frac{s\,u_{m}^*(s)}{s+m}
                                \sum_{b^{\ell-1}\leq n<b^{\ell}} \frac1{n^{m+s}}\;.
\end{equation}
For $s$ real, the partial sums provide alternatingly lower and upper bounds
for $\zeta(s)$.  The general term of this series is, for given complex $s$,
$O(r^m/m)$ with $r=b^{1-\ell}$.
\end{theorem}
\begin{proof}
  The explicit expression in terms of Bernoulli numbers comes from \eqref{eq:uber}.
  There holds by Proposition \ref{prop:bound}:
  \begin{align*}
    \begin{split}
      \bigl|(b^s-b)u_{m}^*(s)\bigr|
      &=  \left|\frac{(s+1)_m}{m!}\right|\bigl|(b^s-b)u_m(s)\bigr|
      \\
&\leq
      \left|\frac{(s+1)_m}{m!}\right|(b^\sigma-b)u_m(\sigma)
      =
      \frac{|(s+1)_m|}{(\sigma+1)_m}(b^\sigma-b)u_m^*(\sigma).
    \end{split}
  \end{align*}
  We have from Proposition \ref{prop:upper} the bound $(b^\sigma-b)u_m^*(\sigma)\leq b^\sigma$.
  The sequence $m\mapsto |(s+1)_m/(\sigma+1)_m|$
  is non-decreasing and, for any complex number $z$, there holds
  $\frac{(z+1)_m}{m!}\sim_{m\to+\infty}\frac{m^z}{\Gamma(z+1)}$ (with $0\sim 0$
  for $z$ a negative integer). We thus get:
  \begin{equation*}
    \bigl|(b^s-b)u_{m}^*(s)\bigr|
    \leq     \frac{|(s+1)_m|}{(\sigma+1)_m}(b^\sigma-b)u_m^*(\sigma)
    \leq
    \frac{\Gamma(\sigma+1)}{|\Gamma(s+1)|}
    b^\sigma\,.
  \end{equation*}
  The lower bound $u_m^*(\sigma)\geq 1/(b^\sigma-b)$ was established in
  Proposition \ref{prop:lower}.

  The recurrence relation \eqref{eq:recum*} is \eqref{eq:recu*} with $N=b$ and $A=\Iffint{0}{b-1}$.

  The series is the one from Theorem \ref{thm:main}. The $O(r^m/m)$ bound with
  $r=b^{1-\ell}$ is a direct consequence of the fact that for a given $s$ the
  sequence $(u_m^*(s))$ is bounded.  The behaviour of partial sums for $s$
  real was pointed out during the proof of Proposition \ref{prop:conv}.
  This concludes the proof.
\end{proof}
\begin{remark}
    For any $\sigma>1$ and any integer $k$, there
  holds $u_m(s_k)= u_m(\sigma)$ with $s_k = \sigma + k 2\pi i/\log b$, hence,
  one has:
  \begin{equation*}
    |u_{m}^*(s_k)|=  \left|\frac{(s_k+1)_m}{m!}\right|u_m(\sigma) 
    >\frac{|(s_k+1)_m|}{(\sigma+1)_m} \frac{1}{b^\sigma - b}\;,
  \end{equation*}
  and
  \begin{equation*}
    \liminf_{m\to\infty}  |u_{m}^*(s_k)| \geq  \frac{\Gamma(\sigma +1)}{|\Gamma(s_k+1)|}
    \frac{1}{b^\sigma - b}\;.
  \end{equation*}
  As $|\Gamma(s_k)|^{-1}$ grows exponentially when $k\to+\infty$ this
  indicates that using the series numerically, even though it is geometrically
  convergent, will require computing a number of terms equal to some affine
  function of the imaginary part of $s$, if a given fixed-point target
  precision is desired.

  And as the series is constructed from the recurrence \eqref{eq:recum*}, it
  appears that the cost for computing $N$ terms is at least quadratic
  in $N$.
\end{remark}

\footnotesize

\providecommand\bibcommenthead{}
\def\blocation#1{\unskip}

\begin{thebibliography}{13}
\ifx \bisbn   \undefined \def \bisbn  #1{ISBN #1}\fi
\ifx \binits  \undefined \def \binits#1{#1}\fi
\ifx \bauthor  \undefined \def \bauthor#1{#1}\fi
\ifx \batitle  \undefined \def \batitle#1{#1}\fi
\ifx \bjtitle  \undefined \def \bjtitle#1{#1}\fi
\ifx \bvolume  \undefined \def \bvolume#1{\textbf{#1}}\fi
\ifx \byear  \undefined \def \byear#1{#1}\fi
\ifx \bissue  \undefined \def \bissue#1{#1}\fi
\ifx \bfpage  \undefined \def \bfpage#1{#1}\fi
\ifx \blpage  \undefined \def \blpage #1{#1}\fi
\ifx \burl  \undefined \def \burl#1{\textsf{#1}}\fi
\ifx \doiurl  \undefined \def \doiurl#1{\url{https://doi.org/#1}}\fi
\ifx \betal  \undefined \def \betal{\textit{et al.}}\fi
\ifx \binstitute  \undefined \def \binstitute#1{#1}\fi
\ifx \binstitutionaled  \undefined \def \binstitutionaled#1{#1}\fi
\ifx \bctitle  \undefined \def \bctitle#1{#1}\fi
\ifx \beditor  \undefined \def \beditor#1{#1}\fi
\ifx \bpublisher  \undefined \def \bpublisher#1{#1}\fi
\ifx \bbtitle  \undefined \def \bbtitle#1{#1}\fi
\ifx \bedition  \undefined \def \bedition#1{#1}\fi
\ifx \bseriesno  \undefined \def \bseriesno#1{#1}\fi
\ifx \blocation  \undefined \def \blocation#1{#1}\fi
\ifx \bsertitle  \undefined \def \bsertitle#1{#1}\fi
\ifx \bsnm \undefined \def \bsnm#1{#1}\fi
\ifx \bsuffix \undefined \def \bsuffix#1{#1}\fi
\ifx \bparticle \undefined \def \bparticle#1{#1}\fi
\ifx \barticle \undefined \def \barticle#1{#1}\fi
\bibcommenthead
\ifx \bconfdate \undefined \def \bconfdate #1{#1}\fi
\ifx \botherref \undefined \def \botherref #1{#1}\fi
\ifx \url \undefined \def \url#1{\textsf{#1}}\fi
\ifx \bchapter \undefined \def \bchapter#1{#1}\fi
\ifx \bbook \undefined \def \bbook#1{#1}\fi
\ifx \bcomment \undefined \def \bcomment#1{#1}\fi
\ifx \oauthor \undefined \def \oauthor#1{#1}\fi
\ifx \citeauthoryear \undefined \def \citeauthoryear#1{#1}\fi
\ifx \endbibitem  \undefined \def \endbibitem {}\fi
\ifx \bconflocation  \undefined \def \bconflocation#1{#1}\fi
\ifx \arxivurl  \undefined \def \arxivurl#1{\textsf{#1}}\fi
\csname PreBibitemsHook\endcsname

\bibitem[\protect\citeauthoryear{Allouche et~al.}{2025}]{alloshalstip2025}
\begin{barticle}
\bauthor{\bsnm{Allouche}, \binits{J.-P.}},
\bauthor{\bsnm{Shallit}, \binits{J.}},
\bauthor{\bsnm{Stipulanti}, \binits{M.}}:
\batitle{Combinatorics on words and generating {D}irichlet series of automatic
  sequences}.
\bjtitle{Discrete Math.}
\bvolume{348}(\bissue{8}),
\bfpage{114487}--\blpage{16}
(\byear{2025})
\doiurl{10.1016/j.disc.2025.114487}
\end{barticle}
\endbibitem

\bibitem[\protect\citeauthoryear{Balazard et~al.}{2005}]{balamendsebb2005}
\begin{barticle}
\bauthor{\bsnm{Balazard}, \binits{M.}},
\bauthor{\bsnm{Mend\`es~France}, \binits{M.}},
\bauthor{\bsnm{Sebbar}, \binits{A.}}:
\batitle{Variations on a theme of {H}ardy's}.
\bjtitle{Ramanujan J.}
\bvolume{9}(\bissue{1-2}),
\bfpage{203}--\blpage{213}
(\byear{2005})
\doiurl{10.1007/s11139-005-0833-5}
\end{barticle}
\endbibitem

\bibitem[\protect\citeauthoryear{Borwein et~al.}{2000}]{borwbradcran2000}
\begin{barticle}
\bauthor{\bsnm{Borwein}, \binits{J.M.}},
\bauthor{\bsnm{Bradley}, \binits{D.M.}},
\bauthor{\bsnm{Crandall}, \binits{R.E.}}:
\batitle{Computational strategies for the {R}iemann zeta function}.
\bjtitle{J. Comput. Appl. Math.}
\bvolume{121}(\bissue{1-2}),
\bfpage{247}--\blpage{296}
(\byear{2000})
\doiurl{10.1016/S0377-0427(00)00336-8}
\end{barticle}
\endbibitem

\bibitem[\protect\citeauthoryear{Burnol}{2025}]{burnolkempner}
\begin{barticle}
\bauthor{\bsnm{Burnol}, \binits{J.-F.}}:
\batitle{Moments in the exact summation of the curious series of {K}empner
  type}.
\bjtitle{Amer. Math. Monthly}
\bvolume{132}(\bissue{10}),
\bfpage{995}--\blpage{1006}
(\byear{2025})
\doiurl{10.1080/00029890.2025.2554555}
\end{barticle}
\endbibitem

\bibitem[\protect\citeauthoryear{Burnol}{2026a}]{burnolirwin}
\begin{barticle}
\bauthor{\bsnm{Burnol}, \binits{J.-F.}}:
\batitle{Measures for the summation of {I}rwin series}.
\bjtitle{Integers}
\bvolume{26} (\bcomment{Paper No. A11}),
\bfpage{1}--\blpage{20}
(\byear{2026})
\doiurl{10.5281/zenodo.18154150}
\end{barticle}
\endbibitem

\bibitem[\protect\citeauthoryear{Burnol}{2026b}]{burnolzeta_umosc}
\begin{botherref}
\oauthor{\bsnm{Burnol}, \binits{J.-F.}}:
Some series representing the {R}iemann zeta function
(2026).
\url{https://arxiv.org/abs/2602.05511}
\end{botherref}
\endbibitem

\bibitem[\protect\citeauthoryear{Hardy}{1907}]{hardy1907}
\begin{barticle}
\bauthor{\bsnm{Hardy}, \binits{G.}}:
\batitle{On certain oscillating series}.
\bjtitle{Quart. J. Pure Appl. Math}
\bvolume{38},
\bfpage{269}--\blpage{288}
(\byear{1907})
\end{barticle}
\endbibitem

\bibitem[\protect\citeauthoryear{Keating and Reade}{2000}]{keatread2000}
\begin{barticle}
\bauthor{\bsnm{Keating}, \binits{J.P.}},
\bauthor{\bsnm{Reade}, \binits{J.B.}}:
\batitle{Summability of alternating gap series}.
\bjtitle{Proc. Edinburgh Math. Soc. (2)}
\bvolume{43}(\bissue{1}),
\bfpage{95}--\blpage{101}
(\byear{2000})
\doiurl{10.1017/S001309150002071X}
\end{barticle}
\endbibitem

\bibitem[\protect\citeauthoryear{K\"ohler and Spilker}{2009}]{kohlspil2009}
\begin{barticle}
\bauthor{\bsnm{K\"ohler}, \binits{G.}},
\bauthor{\bsnm{Spilker}, \binits{J.}}:
\batitle{Dirichlet-{R}eihen zu {K}empners merkw\"urdiger konvergenter {R}eihe}.
\bjtitle{Math. Semesterber.}
\bvolume{56}(\bissue{2}),
\bfpage{187}--\blpage{199}
(\byear{2009})
\doiurl{10.1007/s00591-009-0059-5}
\end{barticle}
\endbibitem

\bibitem[\protect\citeauthoryear{Mend\`es~France and
  Sebbar}{1999}]{mendsebb1999}
\begin{barticle}
\bauthor{\bsnm{Mend\`es~France}, \binits{M.}},
\bauthor{\bsnm{Sebbar}, \binits{A.}}:
\batitle{Pliages de papiers, fonctions th\^eta et m\'ethode du cercle}.
\bjtitle{Acta Math.}
\bvolume{183}(\bissue{1}),
\bfpage{101}--\blpage{139}
(\byear{1999})
\doiurl{10.1007/BF02392948}
\end{barticle}
\endbibitem

\bibitem[\protect\citeauthoryear{Nathanson}{2021}]{nathanson2021jnt}
\begin{barticle}
\bauthor{\bsnm{Nathanson}, \binits{M.B.}}:
\batitle{Dirichlet series of integers with missing digits}.
\bjtitle{J. Number Theory}
\bvolume{222},
\bfpage{30}--\blpage{37}
(\byear{2021})
\doiurl{10.1016/j.jnt.2020.10.002}
\end{barticle}
\endbibitem

\bibitem[\protect\citeauthoryear{Zhang}{2011}]{zhang2011}
\begin{barticle}
\bauthor{\bsnm{Zhang}, \binits{C.}}:
\batitle{La s\'erie enti\`ere
  {$1+\frac{z}{\Gamma(1+i)}+\frac{z^2}{\Gamma(1+2i)}+\frac{z^3}{\Gamma(1+3i)}+\cdots$}
  poss\`ede une fronti\`ere naturelle!}
\bjtitle{C. R. Math. Acad. Sci. Paris}
\bvolume{349}(\bissue{9-10}),
\bfpage{519}--\blpage{522}
(\byear{2011})
\doiurl{10.1016/j.crma.2011.03.010}
\end{barticle}
\endbibitem

\bibitem[\protect\citeauthoryear{Zhang}{2023}]{zhang2023}
\begin{barticle}
\bauthor{\bsnm{Zhang}, \binits{C.}}:
\batitle{Analytical study of the pantograph equation using {J}acobi theta
  functions}.
\bjtitle{J. Approx. Theory}
\bvolume{296},
\bfpage{105974}--\blpage{21}
(\byear{2023})
\doiurl{10.1016/j.jat.2023.105974}
\end{barticle}
\endbibitem

\end{thebibliography}


\end{document}